\documentclass[letterpaper, 10 pt, conference]{ieeeconf}
\IEEEoverridecommandlockouts                             
\usepackage[english]{babel}

\usepackage{verbatim}
\usepackage{amsmath}
\usepackage{amssymb}
\usepackage{graphicx}
\usepackage{subfigure}
\usepackage{hyperref}

 \usepackage[usenames,dvipsnames]{pstricks}
 \usepackage{epsfig}
 \usepackage{pst-grad} 
 \usepackage{pst-plot} 

\title{\LARGE \bf A Graph Partitioning Approach to Predict Patterns in Lateral
Inhibition Systems}

\author{\authorblockN{Ana S. Rufino Ferreira, Murat Arcak}
\authorblockA{Department of Electrical Engineering \& Computer Sciences\\
University of California, Berkeley, CA\\
Emails: ana@eecs.berkeley.edu, arcak@eecs.berkeley.edu}
}

\newtheorem{theorem}{Theorem}[section]

\newtheorem{lemma}[theorem]{Lemma}

\newtheorem{corollary}[theorem]{Corollary}
\newtheorem{assumption}[theorem]{Assumption}

\newcommand{\R}{\mathbb{R}}

\renewcommand{\baselinestretch}{1}

\begin{document}
\maketitle
\thispagestyle{empty}
\pagestyle{empty}
\begin{abstract}
We analyze pattern formation on a network of cells where each cell inhibits its
neighbors through cell-to-cell contact signaling. The network is modeled
as an interconnection of identical dynamical subsystems each of which represents
the signaling reactions in a cell. We search for steady state patterns by
partitioning the graph vertices into disjoint classes, where the cells in the same class have the same final fate.
To prove the existence of steady states with this structure, we use results from
monotone systems theory. Finally, we analyze the stability of these patterns with a
block decomposition based on the graph partition.
\end{abstract}

\section{Introduction}

Spatial patterning plays a crucial role in multicellular
developmental processes \cite{gilbert10,wolpert11}. The majority of
theoretical results on pattern formation rely on
diffusion-driven instabilities, proposed by Turing \cite{turing52} and further
studied by other authors \cite{gierer72, dillon94, murray03}.  
Although some early events in developmental biology employ diffusible
signals, most of the patterning that leads to segmentation and fate-specification relies on
contact-mediated signals.  
Lateral inhibition \cite{sprinzak10} is a
cell-to-cell contact signaling mechanism that leads neighboring cells to
compete and diverge to contrasting states of differentiation. An example of
lateral inhibition is the Notch pathway, where neighboring cells compete
through the interaction between Notch receptors and
Delta ligands \cite{collier96}.\\

Dynamical models for the Notch mechanism have been analyzed in \cite{collier96}
and \cite{sprinzak11}. However, these either include analytical results
for only two cells, or perform numerical simulations for larger networks. A broader dynamical
model for lateral inhibition is proposed in \cite{arcak12}, and results
that are independent of the size of the network are presented. In this
reference, the large-scale network is viewed as an interconnection of individual
cells, each defined by an input-output model.
The contact signaling is represented by an undirected graph, where each vertex is a cell,
and a link between two vertices represents the contact between two cells.
Results for the instability of the homogeneous steady state and the convergence
of two level patterns for bipartite contact graphs are presented in
\cite{arcak12}.\\

In this paper, we use the model introduced in \cite{arcak12} and
derive results for pattern formation on a general contact graph,
recovering the results of \cite{arcak12} for bipartite graphs as a special case.
Our main idea is to partition the graph vertices into disjoint classes, where the cells in the same
class have the same final fate.
We use algebraic properties of the graph and tools from monotone systems theory
\cite{smith95} to prove the existence of steady states that are
patterned according to these partitions. Finally, we address the stability
of these patterns by decomposing the system into two subsystems. The first
describes the dynamics on an invariant subspace defined according to the
partition; and the second describes the dynamics transversal to this subspace.\\

A key property that each partition must satisfy is that the sum of the weights
from one vertex in one class to those in another class is independent of the
choice of vertex. Partitions with this property are called \textit{equitable}
and allow us to study a reduced model where all vertices in the
same class have the same state.
As examples of equitable partitions, we study bipartite graphs and graphs with
symmetries. For symmetric graphs, we show that subgroups of the automorphism
group of a graph can be used to identify equitable partitions.\\

The idea of grouping the vertices of a network into classes of
synchronized states has been explored in \cite{stewart03,golubitsky06}; however,
no investigation about steady-states and their stability is pursued, and no
biological application is addressed.
Symmetry properties have been exploited in the dimension reduction and
block decomposition of semidefinite program optimization problems, such as
fastest mixing Markov chains on the graph \cite{boyd05, boyd09}, and
sum-of-squares \cite{gatermann04}.
Symmetry has also been related to
controllability on controlled agreement problems \cite{rahmani09}.\\

In Section II, we define the model and introduce necessary
graph theoretical concepts. In Section III, we present the main result of the
paper, which provides conditions for the existence of steady-state patterns.
In Sections IV, we apply the main results to graphs with symmetries, and present
examples. Finally, in Section V, we present a decomposition that is helpful
for the stability analysis of steady-state patterns, and a small gain
stability type criterion.

\section{Lateral Inhibition Model}
We represent the cell network by an undirected and connected
graph $\mathcal{G}=\mathcal{G}(V,E)$, where the set of vertices $V$ represents a
group of cells, and each edge $e\in E$ represents a contact
between two cells.
The strength of the contact signal between cells $i$ and $j$ is defined
by the nonnegative constant $w_{i,j}$. We let $w_{i,j}=0$ when $i$ and $j$ are not in
contact and allow uneven weights to represent distinct contact signal
strengths.
This \textit{contact graph} is undirected, \textit{i.e.},
$W=\{w_{ij}\}$ is symmetric.\\
Let $N$ be the number of cells and define the \textit{scaled
adjacency matrix} $P\in\R^{N\times N}$ of $\mathcal{G}$ as:
\begin{equation}\label{eq:adjacencymat}
p_{ij} = 
 d_i^{-1}w_{i,j},
\end{equation}
where the scaling factor is the node degree $d_i=\sum_j w_{i,j}$. 
The
definition of $P$ implies that the matrix is nonnegative and row-stochastic,
\textit{i.e.}, $P\mathbf{1}_N=\mathbf{1}_N$, where $\mathbf{1}_N\in\R^N$ denotes
the vector of ones. The structure of $P$ is identical to the transposed
probability transition matrix of a reversible Markov Chain. Therefore, $P$ has real
valued eigenvalues and eigenvectors.\\

Consider a network of identical cells $i=1,\ldots,N$ whose
dynamical model is given by:
\begin{equation}\label{eq:celldynamics1}
\left\{
\begin{array}{rcl}
\dot{x_i}&=&f(x_i,u_i)\\
y_i&=&h(x_i)
\end{array}
\right.
\end{equation}
where $x_i\in\mathcal{X}\subset\R^n$ describes the state in cell $i$,
$u_i\in\mathcal{U}\subset\R$ is an aggregate input from neighboring cells, and
$y_i\in\mathcal{Y}\subset\R$ represents the output of each cell that contributes
to the input to adjacent cells.\\
We represent the cell-to-cell interaction by
\begin{equation}\label{eq:interconnection}
u=Py
\end{equation}
where $P$ is the scaled adjacency matrix of the contact graph as in
\eqref{eq:adjacencymat}, $u:=[u_{1}\ldots u_{N}]^T$, and $y:=[y_{1}\ldots
y_{N}]^T$. This means that the input to each cell is a weighted average of
the outputs of adjacent cells.\\

We assume that $f(\cdot,\cdot)$ and $h(\cdot)$
are continuously differentiable, and that for each constant input
$u^*\in\mathcal{U}$, system \eqref{eq:celldynamics1} has a globally
asymptotically stable steady-state
\begin{equation}
x^*:=S(u^*).
\end{equation} 
Furthermore, we assume
that the map $S:\mathcal{U}\rightarrow\mathcal{X}$ and map
$T:\mathcal{U}\rightarrow\mathcal{Y}$, defined as:
\begin{equation}\label{eq:staticmodel}
T(\cdot):=h(S(\cdot)),
\end{equation}
are continuously differentiable, and that $T(\cdot)$ is a positive, bounded,
and decreasing function. The decreasing property of $T$
is consistent with the lateral inhibition feature, since higher
outputs in one cell lead to lower values in adjacent cells.\\

Note that the steady states of the system
\eqref{eq:celldynamics1}-\eqref{eq:interconnection} are given by $x_i=S(u_i)$
in which $u_1,\ldots,u_N$ are solutions of the equation:
\begin{equation}\label{eq:patterneq}
u=P\mathbf{T}_N(u),
\end{equation}
where
\begin{equation}
\mathbf{T}_N(u)=[T(u_1),\ldots,T(u_N)]^T.
\end{equation}
Since $P$ is row-stochastic, \eqref{eq:patterneq} admits
the homogeneous solution $u_i=u^*$ for all $i=1,\ldots,N$, where $u^*$ is the
unique fixed point of $T(\cdot)$, \textit{i.e},
\begin{equation}
T(u^*)=u^*.
\end{equation}
\section{Identifying Steady State Patterns}

To identify nonhomogeneous steady states, we introduce the notion of equitable
graph partitions.
For a weighted and undirected
graph $\mathcal{G}(V,E)$ with scaled adjacency matrix $P$, a partition
$\pi$ of the vertex set $V$ into classes $O_1,\ldots,O_{r}$ 
is said to be \textit{equitable} if there exist
$\overline{p}_{ij}$, $i,j=1,\ldots,r$, such that
\begin{equation}\label{eq:overlineP}
\sum_{v\in O_j}p_{uv}=\overline{p}_{ij}\quad \forall u\in O_i.
\end{equation}
This definition is a
modification of \cite[Section 9.3]{godsil01} which considers a partition based
on the weights of the graph $w_{ij}$ instead of the scaled weights $p_{ij}$ in
\eqref{eq:adjacencymat}.\\

The \textit{quotient matrix} $\overline{P}\in\R^{r\times r}$ is formed by the
entries $\overline{p}_{ij}$. It is also a row-stochastic matrix, and its
eigenvalues are a subset of the eigenvalues of $P$, as
can be shown with a slight modification of \cite[Thm9.3.3]{godsil01}.
As we will further discuss, equitable partitions are easy to identify in
bipartite graphs, and in graphs with symmetries.\\

We
search for nonhomogeneous solutions to \eqref{eq:patterneq} in which the entries
corresponding to cells in the same class have the same value. This means that we examine the
reduced set of equations
\begin{equation}\label{eq:reducedequation}
z=\overline{P}\mathbf{T}_r(z),
\end{equation}
where $\overline{P}$ is the quotient matrix of the contact graph $\mathcal{G}$,
and $z\in\R^r$. The patterns determined from the solutions of
\eqref{eq:reducedequation} are structured in such a way that all cells in the
same class have the same fate, \textit{i.e},
\begin{equation}
u_i=z_j \text{ for all } i\in O_j.
\end{equation}\\
We now present a procedure to determine if \eqref{eq:reducedequation} has a
nonhomogeneous solution. Define the \textit{reduced graph}
$\mathcal{G}_{\pi}$ to be a simple graph in which the vertex set is
$\tilde{V}=\{O_1,\ldots,O_r\}$ and the edge set is
\begin{equation}
\tilde{E}=\{(O_i,O_j): i\neq j, \text{ } \overline{p}_{ij}\neq 0 \text{ or }
\overline{p}_{ji}\neq 0\}.
\end{equation}
Note that we omit self-loops in
$\mathcal{G}_{\pi}$ even if $\overline{p}_{ii}\neq 0$.\\

\begin{assumption}\label{asmpt:genbipartite}
The reduced graph $\mathcal{G}_{\pi}$ is bipartite.\\
\end{assumption}

In the following theorem, we determine whether there exists a
solution to the reduced set of equations \eqref{eq:reducedequation} other than
the homogeneous solution $z^*=u^*\mathbf{1}_r$.

\begin{theorem}\label{thm:main}
Let $\pi$ be an equitable partition of the vertices of $\mathcal{G}$ such that
Assumption \ref{asmpt:genbipartite} holds. Let $v_r$ be the eigenvector of
$\overline{P}$ associated with the minimum eigenvalue $\lambda_r$. If $T(\cdot)$
is positive, bounded, and decreasing, and if $T'(u^*)$ is such that
\begin{equation}\label{eq:bifurcationcondition2}
|T'(u^*)|\lambda_r<-1,
\end{equation}
then there exists a solution of \eqref{eq:reducedequation} other than
\begin{equation}
z^*=u^*\mathbf{1}_r.
\end{equation}
\end{theorem}

\begin{proof}
Consider the auxiliary dynamical system:
\begin{equation}\label{eq:quotientsystem}
\left[\begin{array}{c} \dot{z}_1\\ \vdots\\ \dot{z}_r \end{array}
\right] = -\left[\begin{array}{c} z_1\\ \vdots\\ z_r \end{array}
\right] + \overline{P} \left[\begin{array}{c} T(z_1)\\ \vdots\\
T(_{z_r}) \end{array} \right]:=F(z),\text{ }
z\in\R^{r}_{\geq 0}.
\end{equation}
Note that around the homogeneous steady state $z^*$, the Jacobian matrix
\begin{equation}\label{eq:jacobianmat}
DF(z^*) = -I_r+T'(u^*)\overline{P}
\end{equation}
is a nonpositive matrix (since $\overline{P}$ is nonnegative, and from
\eqref{eq:bifurcationcondition2}, $T'(u^*)<0$).\\
We show that under a coordinate transformation the system is
\textit{cooperative}, see \cite[Definition 3.1.3]{smith95}. Following the
bipartite property of $\mathcal{G}_{\pi}$ in Assumption
\ref{asmpt:genbipartite}, we define a partition $J\subset\{1,\ldots,r\}$ and
$J'=\{1,\ldots,r\}\backslash J$ such that no two vertices in the same set are
adjacent. Let $\epsilon_j=0$ if $j\in J$ and $\epsilon_j=1$ if $j\in J'$, and
choose the transformation $Rz$ to be
\begin{equation}
R=diag\{(-1)^{\epsilon_1},\ldots,(-1)^{\epsilon_j},\ldots,(-1)^{\epsilon_r}\}.
\end{equation}
Since the reduced graph is bipartite, $R^{-1}\overline{P}R=R\overline{P}R$ is
a matrix similar to $\overline{P}$ and all of its off-diagonal elements are
nonpositive. In the new coordinates $Rz$, the Jacobian matrix in
\eqref{eq:jacobianmat} becomes $J:=R(DF(z^*))R$ and has nonnegative
off-diagonal elements. This means that the system is cooperative.\\
To prove the existence of a solution $\tilde{z}\neq z^*$ of
\eqref{eq:reducedequation}, we appeal to \cite[Theorem 4.3.3]{smith95} which
stipulates that the largest real part of the eigenvalues of $J$ (designated
as $s(J)$) to be positive with associated eigenvector $v\gg 0$ (\textit{i.e.},
all elements are positive); and that there exists a bounded forward
invariant set.\\
First, note that $J$ is a quasi-positive and irreducible matrix (this is
because the reduced graph is connected, and $T'(u^*)\neq 0$).
Then, we know from \cite[Corollary 4.3.2]{smith95} that there exists an
eigenvector $v\gg 0$ such that $Jv=s(J)v$. For this case, the eigenvalues of $J$
are all real and given by
\begin{equation}\label{eq:evalsreduced}
-1+\lambda_kT'(u^*),\ k=1,\ldots,r,
\end{equation}
where $\lambda_k$ are the eigenvalues of $\overline{P}$. Therefore,
$s(J)=-1+T'(u^*)\lambda_r$. From condition \eqref{eq:bifurcationcondition2} we
conclude that $s(J)>0$ with positive eigenvector $v$, and that $v_r=Rv$ is an
eigenvector of $\overline{P}$ associated with $\lambda_r$ (\textit{i.e.},
$\overline{P}Rv=\lambda_rRv$).\\
Second, since the transformed cooperative system is monotone with respect to
the standard cone $R^r_{\geq 0}$, we conclude that $u^*\mathbf{1}_r+R^r_{\geq
0}$ and $u^*\mathbf{1}_r-R^r_{\geq 0}$ are forward invariant. Furthermore,
since $T(\cdot)$ is bounded and decreasing (and $T(u^*)=u^*$), there exists an
hypercube $[0,\overline{u}]^{r}$, with $0<u^*<\overline{u}$ which is also
forward invariant. This can be seen from the fact that at $z=0$,
$\dot{z}=\overline{P}T(0)\geq 0$, and at $z=\overline{u}$, $\dot{z}\leq 0$
(since $\overline{u}>T(\overline{u})$). The sets
\begin{equation}
S_1=(u^*\mathbf{1}+R^r_{\geq 0})\cap[0,\overline{u}]^{r} \text{ and }
S_2=(u^*\mathbf{1}-R^r_{\geq 0})\cap[0,\overline{u}]^{r},
\end{equation}
are forward invariant. Therefore, we conclude from \cite[Theorem
4.3.3]{smith95}, there exists an equilibrium point $\tilde{z}\neq z^*$, and it
satisfies \eqref{eq:reducedequation}.
\end{proof}

\vspace{0.5cm}
\textit{Example 1: Checkerboard Patterns in Bipartite Graphs}\\
Suppose that the contact graph $\mathcal{G}$ is bipartite, and choose $O_1$ and
$O_2$ to be the partition such that every edge can only connect a vertex in
$O_1$ to a vertex in $O_2$. Then, up to vertex relabeling, the scaled adjacency
matrix of $\mathcal{G}$ can be written as
\begin{equation}
P=\left[\begin{array}{cc}
0 & P_{12}\\
P_{21} & 0
\end{array}
\right].
\end{equation}
Since the rows of $P_{12}$ (and also rows of $P_{21}$) sum up to $1$, we
conclude that $\pi$, consisting of sets $O_1$ and $O_2$, is
an equitable partition.
Moreover, the reduced graph $\mathcal{G}_{\pi}$ is itself bipartite
(\textit{i.e.}, Assumption \ref{asmpt:genbipartite} holds), and matrix
$\overline{P}$ is given by
\begin{equation}
\overline{P}=\left[\begin{array}{cc} 0 & 1\\ 1 & 0 \end{array} \right].
\end{equation}
Since the eigenvalues/eigenvectors of
$\overline{P}$ are $\lambda_1=1$ and
$\lambda_2=-1$, the next Corollary follows.\\

\begin{corollary}\label{cor:mainBipartite}
Let $\mathcal{G}$ be bipartite, and define
a partition $O_1\subset\{1,\ldots,N\}$ and $O_2=\{1,\ldots,N\}\backslash O_1$
such that no two vertices in the same set are adjacent. Then, if 
\begin{equation}
|T'(u^*)|>1,
\end{equation}
there exists a steady state $u=[u_1\ldots u_N]$ such
that $u_i=z_1$ if $i\in O_1$, and $u_i=z_2$ if $i\in O_2$, and
$z_1\neq z_2\neq u^*$.\hfill $\blacksquare$
\end{corollary}

\vspace{0.5cm}
The steady state defined by Corollary \ref{cor:mainBipartite} results in a
``checkerboard'' pattern as in Figure \ref{fig:twomeshresults}(A), since it has
distinct states for adjacent cells.

\section{Graphs with Symmetries}
An important class of equitable partitions results from graph
symmetries, which are formalized with the notion of graph automorphisms.
For a weighted graph $\mathcal{G}(V,E)$, an \textit{automorphism} is a
permutation $g:V\rightarrow V$ such that if $(i,j)\in E$ then also $(gi,gj)\in
E$ and $w_{i,j}=w_{gi,gj}$,  where $gi$ denotes the image of
vertex $i$ under permutation $g$. The
set of all automorphisms forms a group designated by \textit{automorphism
group}, $Aut(\mathcal{G})$.
A subset $H$ of a full automorphism group $Aut(\mathcal{G})$ is called a
\textit{subgroup} if $H$ is closed under composition and inverse.\\
Let
$H$ be a subgroup of a full automorphism group $Aut(\mathcal{G})$. Then, the
action of all permutations $h\in H$ forms a partition of the vertex
set into \textit{orbits},
$O_i=\{hi:h\in H\}$, such that $O_j=O_i$ for all $j\in O_i$. Let $r$ be
the number of distinct orbits under the subgroup $H$, and relabel them as
$\{\overline{O}_1,\ldots,\overline{O}_{r}\}$. This orbit partition is
equitable, because the sum $\sum_{k\in
\overline{O}_j}p_{i^*k}=\overline{p}_{ij}$ is constant independently of the
choice of $i^*\in\overline{O}_i$.\\
Since any subgroup of the full automorphism group of a
graph leads to an equitable partition, we conclude by Theorem \ref{thm:main}
that any orbit partition generated by a subgroup of $Aut(\mathcal{G})$ is a
candidate for a pattern structured according to this partition.\\

\textit{Example 2: Two-Dimensional Mesh}\\
Consider a two-dimensional mesh with wraparounds
as in Figure \ref{fig:twodimmesh}. Since the graph is bipartite, an equitable partition is
given by the two disjoint subsets of vertices $O_1=\{1,3,6,8,9,11,14,16\}$, and
$O_2=\{2,4,5,7,10,12,13,15\}$. From Corollary \ref{cor:mainBipartite}, we know
that a pattern with final value $u_{1}$ for all cells in $O_1$, and
$u_{2}$ for all cells in $O_2$, with $u_{1}\neq
u_{2}\neq u^*$, is a steady state of the network when
$|T'(u^*)|>1$; see Figure \ref{fig:twomeshresults} (A).\\
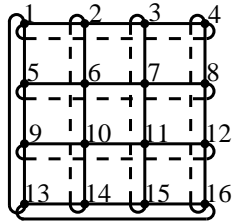
\begin{figure}[ht]
\centering
\scalebox{1} 
{
\begin{pspicture}(0,-1.47)(3.02,1.49)
\psline[linewidth=0.04cm](0.2,1.15)(2.6,1.15)
\psline[linewidth=0.04cm](0.2,1.15)(0.2,-1.25)
\psline[linewidth=0.04cm](0.2,-1.25)(2.6,-1.25)
\psline[linewidth=0.04cm](0.2,-0.45)(2.6,-0.45)
\psline[linewidth=0.04cm](0.2,0.35)(2.6,0.35)
\psline[linewidth=0.04cm](2.6,1.15)(2.6,-1.25)
\psline[linewidth=0.04cm](1.8,-1.25)(1.8,1.15)
\psline[linewidth=0.04cm](1.0,-1.25)(1.0,1.15)
\rput{-90.0}(1.57,3.67){\psarc[linewidth=0.04](2.62,1.05){0.1}{0.0}{180.0}}
\rput{-270.0}(1.23,0.85){\psarc[linewidth=0.04](0.19,1.04){0.09}{0.0}{180.0}}
\psline[linewidth=0.04cm,linestyle=dashed,dash=0.16cm 0.16cm](0.2,0.95)(2.6,0.95)
\usefont{T1}{ptm}{m}{n}
\rput(0.27,1.295){1}
\usefont{T1}{ptm}{m}{n}
\rput(1.15,1.295){2}
\usefont{T1}{ptm}{m}{n}
\rput(1.94,1.295){3}
\usefont{T1}{ptm}{m}{n}
\rput(2.73,1.275){4}
\usefont{T1}{ptm}{m}{n}
\rput(0.32,0.495){5}
\usefont{T1}{ptm}{m}{n}
\rput(1.13,0.495){6}
\usefont{T1}{ptm}{m}{n}
\rput(1.93,0.495){7}
\usefont{T1}{ptm}{m}{n}
\rput(2.71,0.495){8}
\usefont{T1}{ptm}{m}{n}
\rput(0.35,-0.305){9}
\usefont{T1}{ptm}{m}{n}
\rput(1.18,-0.305){10}
\usefont{T1}{ptm}{m}{n}
\rput(1.96,-0.305){11}
\usefont{T1}{ptm}{m}{n}
\rput(2.78,-0.305){12}
\usefont{T1}{ptm}{m}{n}
\rput(0.37,-1.105){13}
\usefont{T1}{ptm}{m}{n}
\rput(1.18,-1.105){14}
\usefont{T1}{ptm}{m}{n}
\rput(1.97,-1.105){15}
\usefont{T1}{ptm}{m}{n}
\rput(2.78,-1.105){16}
\psdots[dotsize=0.12](0.2,1.15)
\psdots[dotsize=0.12](1.0,1.15)
\psdots[dotsize=0.12](1.8,1.15)
\psdots[dotsize=0.12](2.6,1.15)
\psdots[dotsize=0.12](0.2,0.35)
\psdots[dotsize=0.12](1.0,0.35)
\psdots[dotsize=0.12](1.8,0.35)
\psdots[dotsize=0.12](2.6,0.35)
\psdots[dotsize=0.12](0.2,-0.45)
\psdots[dotsize=0.12](1.0,-0.45)
\psdots[dotsize=0.12](1.8,-0.45)
\psdots[dotsize=0.12](2.6,-0.45)
\psdots[dotsize=0.12](0.2,-1.25)
\psdots[dotsize=0.12](1.0,-1.25)
\psdots[dotsize=0.12](1.8,-1.25)
\psdots[dotsize=0.12](2.6,-1.25)
\rput{-90.0}(2.37,2.87){\psarc[linewidth=0.04](2.62,0.25){0.1}{0.0}{180.0}}
\rput{-270.0}(0.43,0.05){\psarc[linewidth=0.04](0.19,0.24){0.09}{0.0}{180.0}}
\psline[linewidth=0.04cm,linestyle=dashed,dash=0.16cm 0.16cm](0.2,0.15)(2.6,0.15)
\rput{-90.0}(3.17,2.07){\psarc[linewidth=0.04](2.62,-0.55){0.1}{0.0}{180.0}}
\rput{-270.0}(-0.37,-0.75){\psarc[linewidth=0.04](0.19,-0.56){0.09}{0.0}{180.0}}
\psline[linewidth=0.04cm,linestyle=dashed,dash=0.16cm 0.16cm](0.2,-0.65)(2.6,-0.65)
\rput{-90.0}(3.97,1.27){\psarc[linewidth=0.04](2.62,-1.35){0.1}{0.0}{180.0}}
\rput{-270.0}(-1.17,-1.55){\psarc[linewidth=0.04](0.19,-1.36){0.09}{0.0}{180.0}}
\psline[linewidth=0.04cm](0.16,-1.45)(2.62,-1.45)
\rput{-180.0}(5.02,-2.52){\psarc[linewidth=0.04](2.51,-1.26){0.1}{0.0}{180.0}}
\psarc[linewidth=0.04](2.5,1.17){0.09}{0.0}{180.0}
\psline[linewidth=0.04cm,linestyle=dashed,dash=0.16cm 0.16cm](2.41,1.16)(2.41,-1.24)
\rput{-180.0}(3.42,-2.52){\psarc[linewidth=0.04](1.71,-1.26){0.1}{0.0}{180.0}}
\psarc[linewidth=0.04](1.7,1.17){0.09}{0.0}{180.0}
\psline[linewidth=0.04cm,linestyle=dashed,dash=0.16cm 0.16cm](1.61,1.16)(1.61,-1.24)
\rput{-180.0}(1.82,-2.52){\psarc[linewidth=0.04](0.91,-1.26){0.1}{0.0}{180.0}}
\psarc[linewidth=0.04](0.9,1.17){0.09}{0.0}{180.0}
\psline[linewidth=0.04cm,linestyle=dashed,dash=0.16cm 0.16cm](0.81,1.16)(0.81,-1.24)
\rput{-180.0}(0.21,-2.53){\psarc[linewidth=0.04](0.105,-1.265){0.105}{0.0}{180.0}}
\psarc[linewidth=0.04](0.095,1.175){0.095}{0.0}{180.0}
\psline[linewidth=0.04cm](0.0,1.19)(0.0,-1.27)
\end{pspicture} 
}
\caption{Graph representation for a two-dimensional mesh with wraparounds.}
\label{fig:twodimmesh}
\end{figure}

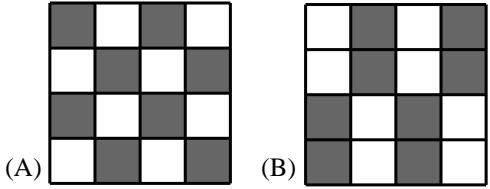
\begin{figure}[ht]
\centering
\scalebox{1} 
{
\begin{pspicture}(0,-1.23)(6.42,1.23)
\definecolor{color647b}{rgb}{0.4,0.4,0.4}
\psframe[linewidth=0.04,linecolor=color647b,dimen=outer,fillstyle=solid,fillcolor=color647b](1.8,-0.59)(1.2,-1.19)
\psframe[linewidth=0.04,linecolor=color647b,dimen=outer,fillstyle=solid,fillcolor=color647b](1.2,0.01)(0.6,-0.59)
\psframe[linewidth=0.04,linecolor=color647b,dimen=outer,fillstyle=solid,fillcolor=color647b](2.4,1.21)(1.8,0.61)
\psframe[linewidth=0.04,linecolor=color647b,dimen=outer,fillstyle=solid,fillcolor=color647b](3.0,0.61)(2.4,0.01)
\psframe[linewidth=0.04,linecolor=color647b,dimen=outer,fillstyle=solid,fillcolor=color647b](2.4,0.01)(1.8,-0.59)
\psframe[linewidth=0.04,linecolor=color647b,dimen=outer,fillstyle=solid,fillcolor=color647b](3.0,-0.59)(2.4,-1.19)
\psframe[linewidth=0.04,linecolor=color647b,dimen=outer,fillstyle=solid,fillcolor=color647b](1.2,1.21)(0.6,0.61)
\psframe[linewidth=0.04,linecolor=color647b,dimen=outer,fillstyle=solid,fillcolor=color647b](1.8,0.61)(1.2,0.01)
\psline[linewidth=0.04cm](0.6,1.21)(3.0,1.21)
\psline[linewidth=0.04cm](0.6,1.21)(0.6,-1.19)
\psline[linewidth=0.04cm](1.2,1.21)(1.2,-1.19)
\psline[linewidth=0.04cm](1.8,1.21)(1.8,-1.19)
\psline[linewidth=0.04cm](2.4,1.21)(2.4,-1.19)
\psline[linewidth=0.04cm](3.0,1.21)(3.0,-1.19)
\psline[linewidth=0.04cm](0.6,0.61)(3.0,0.61)
\psline[linewidth=0.04cm](0.6,0.01)(3.0,0.01)
\psline[linewidth=0.04cm](0.6,-0.59)(3.0,-0.59)
\psline[linewidth=0.04cm](3.0,-1.19)(0.6,-1.19)
\usefont{T1}{ptm}{m}{n}
\rput(0.25,-1.005){(A)}
\usefont{T1}{ptm}{m}{n}
\rput(3.64,-1.005){(B)}
\psframe[linewidth=0.04,linecolor=color647b,dimen=outer,fillstyle=solid,fillcolor=color647b](4.6,-0.61)(4.0,-1.21)
\psframe[linewidth=0.04,linecolor=color647b,dimen=outer,fillstyle=solid,fillcolor=color647b](4.6,-0.01)(4.0,-0.61)
\psframe[linewidth=0.04,linecolor=color647b,dimen=outer,fillstyle=solid,fillcolor=color647b](5.8,-0.61)(5.2,-1.21)
\psframe[linewidth=0.04,linecolor=color647b,dimen=outer,fillstyle=solid,fillcolor=color647b](5.8,-0.01)(5.2,-0.61)
\psframe[linewidth=0.04,linecolor=color647b,dimen=outer,fillstyle=solid,fillcolor=color647b](6.4,0.59)(5.8,-0.01)
\psframe[linewidth=0.04,linecolor=color647b,dimen=outer,fillstyle=solid,fillcolor=color647b](6.4,1.19)(5.8,0.59)
\psframe[linewidth=0.04,linecolor=color647b,dimen=outer,fillstyle=solid,fillcolor=color647b](5.2,0.59)(4.6,-0.01)
\psframe[linewidth=0.04,linecolor=color647b,dimen=outer,fillstyle=solid,fillcolor=color647b](5.2,1.19)(4.6,0.59)
\psline[linewidth=0.04cm](4.0,1.19)(6.4,1.19)
\psline[linewidth=0.04cm](4.0,1.19)(4.0,-1.21)
\psline[linewidth=0.04cm](4.6,1.19)(4.6,-1.21)
\psline[linewidth=0.04cm](5.2,1.19)(5.2,-1.21)
\psline[linewidth=0.04cm](5.8,1.19)(5.8,-1.21)
\psline[linewidth=0.04cm](6.4,1.19)(6.4,-1.21)
\psline[linewidth=0.04cm](4.0,0.59)(6.4,0.59)
\psline[linewidth=0.04cm](4.0,-0.01)(6.4,-0.01)
\psline[linewidth=0.04cm](4.0,-0.61)(6.4,-0.61)
\psline[linewidth=0.04cm](6.4,-1.21)(4.0,-1.21)
\end{pspicture} 
}
\caption{Patterns for the graph in Figure \eqref{fig:twodimmesh}.}
\label{fig:twomeshresults}
\end{figure}

We next consider the automorphism subgroup that
is generated by a combination of two cell rotations in the horizontal direction, one cell
rotation in the vertical direction, and one cell rotation in both vertical
and horizontal directions. This subgroup leads to the orbits
$O_1=\{1,3,5,7,10,12,14,16\}$, $O_2=\{2,4,6,8,9,11,13,15\}$. The
quotient matrix associated with this partition is given by
\begin{equation*}
\overline{P}_B=\left[\begin{array}{cc}
\tfrac{1}{4} & \tfrac{3}{4}\\ \tfrac{3}{4} & \tfrac{1}{4}
\end{array}\right],
\end{equation*}
and has eigenvalues $-1/2$ and $1$. Therefore, from Theorem \ref{thm:main}, a
steady state state given by $z=\overline{P}\mathbf{T}(z)$ exists if
$|T'(u^*)|>2$.\\

In this example, the equitable partition obtained from the
bipartite property of the contact graph can also be obtained by a subgroup of
the automorphism group of the graph. However, this is in general not true;
the four cell path is an example of a graph with a bipartite partition that
cannot be defined by an orbit partition.\\

The computation of automorphism groups, and the identification of the reduced
order systems, becomes cumbersome as the size and symmetries of the graphs
increase.
However, these can be obtained from a computer algebra system with
emphasis on computational group theory, such as GAP, \cite{gap4}.\\

\textit{Example 3: Two-dimensional Hexagonal Cyclic Lattice}

The number of distinct equitable partitions in a hexagonal lattice of cells is
considerably large \cite{wang05}. We use computational algebra algorithms to
find all the possible two-level equitable partitions obtained by automorphisms
subgroups in a $6\times 6$ cyclic lattice. Five distinct partitions, each with
two classes, are plotted in Figure \ref{fig:hexagonal}.
 \begin{figure}[ht]
\centering
	\includegraphics[width=0.50\textwidth]{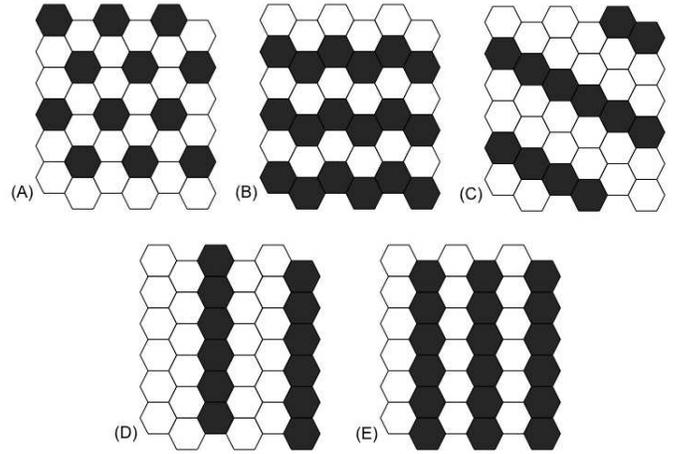}
	\caption{Five distinct partitions, each with two classes, in a $6\times 6$
	Two-Dimensional Cyclic Hexagonal Lattice obtained by symmetries on the
	contact graph.}\label{fig:hexagonal}
\end{figure}

For these partitions, we have the following scaled adjacency graph of the
auxiliary system:
\begin{equation*}
\overline{P}_A=\left[\begin{array}{cc} 0 & 1\\ \tfrac{1}{2} & \tfrac{1}{2}
\end{array} \right], \overline{P}_{B}=\left[\begin{array}{cc} \tfrac{1}{3} &
\tfrac{2}{3}\\ \tfrac{2}{3} & \tfrac{1}{3} \end{array} \right],
\overline{P}_{C}=\left[\begin{array}{cc} \tfrac{1}{3} &
\tfrac{2}{3}\\ \tfrac{1}{3} & \tfrac{2}{3} \end{array} \right],
\end{equation*}
and $\overline{P}_D=\overline{P}_C$, $\overline{P}_E=\overline{P}_B$. For each
matrix we have the following smallest eigenvalue $\lambda_A=-1/2$,
$\lambda_B=\lambda_E=-1/3$, and $\lambda_C=\lambda_D=0$. We thus conclude from
Theorem \ref{thm:main} that the steady-state represented by pattern $A$ in
Figure \ref{fig:hexagonal} exists when $|T'(u^*)|>2$, and patterns $B,E$ are
steady-states when $|T'(u^*)|>3$. Theorem \ref{thm:main} is inconclusive for
patterns $C,D$.\\

\textit{Example 4: Soccerball Pattern on a Buckminsterfullerene Graph}

The next example addresses a larger graph, with 32 cells. It is motivated by the
truncated icosahedron solid, also known as the \textit{Buckminsterfullerene}
\cite{godsil01}, formed by 12 regular pentagonal faces, and 20 regular hexagonal
faces, see Figure \ref{fig:soccerball}. In this case, we assume that each face is a vertex
and that two vertices are connected if the corresponding faces have a common
edge.

\begin{figure}[ht]
\centering
  	\includegraphics[width=0.35\textwidth]{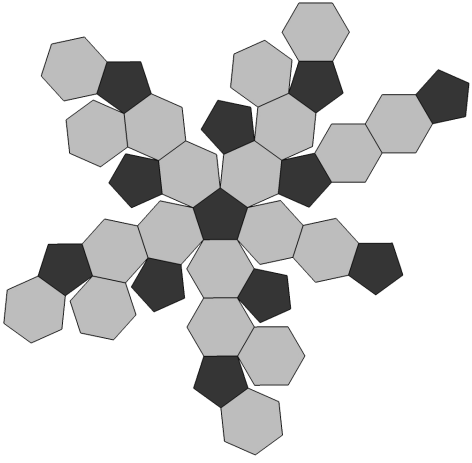}
  \caption{Cell network and soccer ball pattern for the
  Buckminsterfullerene graph.}
\label{fig:soccerball}
\end{figure}

The full automorphism group leads to two orbits, one
that consists of all the regular pentagon cells ($O_P$), and the second orbit
encloses all the regular hexagon cells ($O_H$). The quotient matrix associated
with the orbit partition is then
\begin{equation}
\overline{P}=\left[\begin{array}{cc}
0&1\\1/2&1/2
\end{array} \right].
\end{equation}
This matrix has eigenvalues $1$ and $-1/2$. Therefore, we conclude from Theorem
\ref{thm:main}, that a steady state as in Figure \ref{fig:soccerball} exists
when $|T'(u^*)|>2$.\\

\textit{Example 5: Nonbipartite nor Symmetric Equitable Partition}\\
As discussed above, both bipartitions and automorphism subgroups (symmetries)
lead to equitable partitions. However, these are not the only cases that lead
to equitable partitions. Consider the graph in Figure \ref{fig:equitablepart}
with partition $C_1=\{3,6\}$, and $C_2=\{1,2,4,5,7,8\}$. This partition is
equitable, but it does not result from an automorphism subgroup (for instance,
there is no automorphism exchanging vertices $1$ and $4$), and the graph
is also not bipartite (due to the odd length cycles).

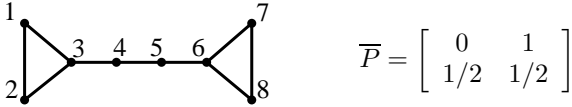
\begin{figure}[ht]
  \begin{minipage}{0.24\textwidth}
\centering
\scalebox{1} 
{
\begin{pspicture}(0,-0.72)(3.56,0.74)
\psdots[dotsize=0.12](0.88,-0.14)
\psdots[dotsize=0.12](3.28,0.38)
\psdots[dotsize=0.12](2.68,-0.14)
\psdots[dotsize=0.12](3.28,-0.64)
\psdots[dotsize=0.12](1.48,-0.14)
\psdots[dotsize=0.12](2.08,-0.14)
\psdots[dotsize=0.12](0.26,0.38)
\psdots[dotsize=0.12](0.26,-0.64)
\psline[linewidth=0.04cm](0.88,-0.14)(1.48,-0.14)
\psline[linewidth=0.04cm](1.48,-0.14)(2.08,-0.14)
\psline[linewidth=0.04cm](2.08,-0.14)(2.68,-0.14)
\psline[linewidth=0.04cm](3.28,0.36)(3.28,-0.62)
\psline[linewidth=0.04cm](0.26,0.36)(0.26,-0.62)
\psline[linewidth=0.04cm](0.88,-0.14)(0.24,0.4)
\psline[linewidth=0.04cm](3.28,-0.64)(2.68,-0.14)
\psline[linewidth=0.04cm](0.86,-0.16)(0.26,-0.64)
\psline[linewidth=0.04cm](3.28,0.38)(2.68,-0.14)
\usefont{T1}{ptm}{m}{n}
\rput(0.07,0.545){1}
\usefont{T1}{ptm}{m}{n}
\rput(0.09,-0.495){2}
\usefont{T1}{ptm}{m}{n}
\rput(0.98,0.045){3}
\usefont{T1}{ptm}{m}{n}
\rput(1.53,0.065){4}
\usefont{T1}{ptm}{m}{n}
\rput(2.02,0.065){5}
\usefont{T1}{ptm}{m}{n}
\rput(2.57,0.065){6}
\usefont{T1}{ptm}{m}{n}
\rput(3.43,0.525){7}
\usefont{T1}{ptm}{m}{n}
\rput(3.43,-0.495){8}
\end{pspicture} 
}
  \end{minipage}
  \begin{minipage}{0.24\textwidth}
    \begin{equation*}
     \overline{P}=\left[\begin{array}{cc} 0 & 1\\ 1/2 & 1/2 \end{array}\right]
    \end{equation*}
  \end{minipage}
\caption{Example of an equitable partition that is neither bipartite nor a
symmetry ($C_1=\{3,6\}$, $C_2=\{1,2,4,5,7,8\}$).}
\label{fig:equitablepart}
\end{figure}

The quotient matrix is the same as in Example 4, we conclude that a two level
steady state pattern formed by cells $C_1$ and $C_2$ exists when $|T'(u^*)|>2$.\\

\section{Stability Analysis of Patterns}
In the previous sections we discussed the existence of nonhomogeneous
steady states for the cell network
\eqref{eq:celldynamics1}-\eqref{eq:interconnection}.
Determining stability for these steady states may become cumbersome for a large
network of cells. To simplify this task, we decompose
the system into an appropriate interconnection of lower order subsystems, and
make the interconnection structure explicit.\\

\subsection{Network Decomposition}
Note that an equitable partition defines the subspace 
$x_i=x_j$ for vertices $i$ and $j$ in the
same class. This subspace is invariant for the full system
\eqref{eq:celldynamics1}-\eqref{eq:interconnection}.
Therefore, the
steady states identified using an equitable partition of the contact graph lie
on the corresponding invariant subspace. For a partition of
dimension $r$, ($O_1,\ldots,O_r$), the reduced order dynamics on this subspace
consists of $r$ subsystems as defined in \eqref{eq:celldynamics1}, coupled by
$u=\overline{P}y$.\\
Let the steady state be defined by
$x_i=S(u_i)$, $i=1,\ldots,N$, where $u_i=z_j$ if $i\in O_j$ and
$[z_1,\ldots,z_r]$ is a solution of \eqref{eq:reducedequation}. The
linearization at this steady state has the form
\begin{equation}\label{eq:linearization}
\dot{x}=(A+BPC)x,
\end{equation}
where $A\in\R^{Nn\times Nn}$ is a block diagonal matrix where the $i$-th block
is equal to $A_j$ if $i$ is in orbit $O_j$, and $A_j\in\R^{n\times n}$,
is given by:
\begin{equation}\label{eq:linearization2}
A_j=\frac{\partial f(x,u)}{\partial x}|_{(x_j,u_j)}.
\end{equation}
Similarly, $B\in\R^{Nn\times N}$ and $C\in\R^{N\times Nn}$ are block diagonal
matrices as in $A$, with $B_j=\frac{\partial f(x,u)}{\partial
u}|_{(x_j,u_j)}$, and $C_j=\frac{\partial h(x)}{\partial x}|_{x_j}$.\\
To decompose \eqref{eq:linearization} into two subsystems, we select a
representative vertex $V_i$ for each class $O_i$. The set of $r$
representatives of each class defines the state of the subsystem on the
invariant subspace. To see this, let $Q$ be a matrix in $\R^{N\times r}$, where
$q_{ij}=1$ if cell $i$ is in class $j$, and $q_{ij}=0$ otherwise. Since the
partition is equitable, we conclude that
\begin{equation}
PQ=Q\overline{P}.
\end{equation}
Letting $T:=[Q\ R]$ and choosing $R$
to be a matrix in $\R^{N\times (N-r)}$ with
columns that, together with those of $Q$, form a basis for $\R^N$, we conclude
that there exist matrices $C$ and $M$ such that
\begin{equation}\label{eq:blockPmatrix}
P[\begin{array}{cc} Q & R\end{array}]=
[\begin{array}{cc} Q & R\end{array}]\left[
\begin{array}{cc}\overline{P} & C \\ 0 & M\end{array} 
\right]:=T\tilde{P}.
\end{equation}
The matrix $T$ is invertible and, thus, defines a similarity transformation from
matrix $P$ to $\tilde{P}$. Note that the upper left diagonal block of $\tilde{P}$ is the
matrix $\overline{P}$, which describes the reduced order subsystem
defined by the representative vertices.\\
Next, we study a particular choice of the matrix $R$ that gives a
meaningful variable representation to the transverse subspace dynamics.
Let the columns of $R$ be given by
standard vectors $e_i$, defined as $e_{ij}=\delta_{ij}$, $j=1,\ldots,N$; and
further select the columns of $R$ to be $e_i$, $i\in\{O_1\backslash
V_1,\ldots,O_r\backslash V_r\}$, in such a way that if $i\in O_p\backslash
V_p$, $j\in O_q\backslash V_q$, and $p<q$, then $e_i$ is in a column before
$e_j$, \textit{i.e.}, the column with non-zero entry $i$ is placed before the
column with entry $j$ if vertex $i$ is in a class with smaller index than the class of vertex $j$.\\
For this choice of $R$ we conclude from \cite[Section 5.3]{cvetkovic80} that
the change of variables $T\otimes I_n$, \textit{i.e.}, $\tilde{x}=(T^{-1}\otimes
I_n) x$, leads to the decomposition of the linearized dynamics into the
subsystem of the $r$ representative cells $V_i$ defining the invariant
subspace, and the transverse subspace that is defined by the state difference
between all the other cells in a class and their corresponding class representative:
\begin{equation}
\tilde{x}_i = \left\{\begin{array}{rl}
x_{V_i}, & i=1,\ldots,r\\
x_k-x_{V_j}, & i=r+1,\ldots,N
\end{array}\right. , \text{ where } k\in O_j\backslash V_j.
\end{equation}
Therefore, the linearized system is decomposed
into the representative subsystem $S_R$ and the transverse subsystem $S_D$,
\begin{equation}\label{eq:sysdecomp}
S_R=\{A_R,B_R,C_R\}\ \text{ and }\ S_D=\{A_D,B_D,C_D\}.
\end{equation}
To see this, note that
\begin{equation}
\dot{\tilde{x}}=(T^{-1}\otimes I_n) (A+BPC)(T\otimes I_n)\tilde{x} =
(\tilde{A}+\tilde{B}\tilde{P}\tilde{C})\tilde{x},
\end{equation}
where
\begin{eqnarray}
\tilde{A}&:=&T^{-1}\otimes I_n AT\otimes I_n\nonumber\\
&=&diag\{A_1,A_2,\ldots,A_r,A_1,\ldots\nonumber\\
&&\hspace{1.8cm}\ldots,A_1,A_2,\ldots,A_2,\ldots,A_r,\ldots,A_r\}\nonumber\\
&:=&diag\{A_R,
A_D\},
\end{eqnarray}
with $A_R\in\R^{rn\times rn}$, $A_D\in\R^{(N-r)n\times(N-r)n}$, and matrices
$\tilde{B}$ and $\tilde{C}$ are defined similarly to $\tilde{A}$. This means
that to prove the stability of \eqref{eq:linearization} we need to
guarantee the stability of matrices
\begin{equation}\label{eq:repsys}
A_R+B_R\overline{P}C_R\ \text{ and }\ A_D+B_DMC_D.
\end{equation}
Stability certification of the steady state pattern is thus simplified due to
the fact that $\overline{P}$ is typically much smaller than $P$, and $M$ can be further
decomposed by using a systematic approach based on the nested hierarchy of
automorphism groups of the contact graph $\mathcal{G}$.

\subsection{A Small-Gain Criterion for Stability}
In this section, we provide a small-gain type condition for the stability of the
steady state pattern around the solutions of \eqref{eq:patterneq}, based on
the solution $[z_1,\ldots,z_r]$ of \eqref{eq:reducedequation}, which is defined
by
\begin{equation}\label{eq:SteadyState} 
x_i=S(u_i)\text{, }\ i=1,\ldots,N\text{, }\ \text{where }u_i=z_j\ \text{ if }\ i\in O_j.
\end{equation}\\
Consider again the linearization introduced in \eqref{eq:linearization2}. This
describes a cell network with interconnection defined by $u=Py$ and
where each individual cell is given by the following linearized subsystem:
\begin{equation}\label{eq:linearizedcell}
H_i:\left\{
\begin{array}{rcl}
\dot{x}^i&=&A^ix^i+B^iu^i\\
y^i&=&C^ix^i
\end{array}
\right. ,
\end{equation}
where $A^i=A_j$, if $i\in O_j$ as in \eqref{eq:linearization2}, and similarly
for $B^i$ and $C^i$, see figure \ref{fig:linearinterconnection}. Assume that
each linearized subsystem $H_i$ is observable and that $A^i$ is Hurwitz.\\

\begin{figure}[ht]
\centering
  	\includegraphics[width=0.3\textwidth]{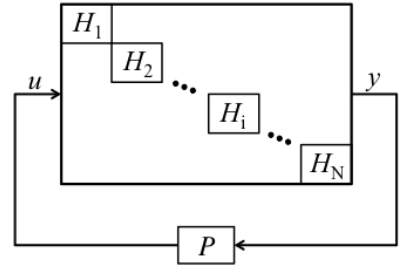}
  \caption{Linearized System Interconnection.}
\label{fig:linearinterconnection}
\end{figure}

Since subsystems in the same class have identical models, they have identical
$\mathcal{L}_2$-gains. Let $\gamma_i$ denote the $\mathcal{L}_2$-gain of each
subsystem in class $i$, and let $\Gamma$ be a diagonal matrix with entries
\begin{equation}\label{eq:gamma}
\{\Gamma\}_{jj}=\gamma_i\ \text{ for }\ j\in O_i.
\end{equation}
The following Proposition
provides a small-gain criterion for the stability of the cell network around the
steady state pattern.\\

\begin{theorem}\label{thm:stabilitytheorem}
Consider the network \eqref{eq:celldynamics1}-\eqref{eq:interconnection}. The
steady state pattern defined by \eqref{eq:SteadyState} is
locally asymptotically stable if
\begin{equation}\label{eq:smallgaincondition}
\rho(P\Gamma)<1,
\end{equation}
where $\Gamma$ is as in \eqref{eq:gamma}.
\end{theorem}

\begin{proof}
Since each linearized subsystem $H_i$, in \eqref{eq:linearizedcell}, has bounded
$\mathcal{L}_2$ gain, by the Bounded Real Lemma \cite{doyle89}, we conclude that
there exists a positive definite matrix $Q_i$ such that, for $V_i(z)=z^TQ_iz$,
we have $\dot{V}_i(x^i,u^i)\leq \gamma_i^2 u^{iT} u^i - y^{iT} y^i$. Let
$D=diag\{d_1,\ldots,d_n\}$, where $d_i$ is some positive constant, and let
$V(x)=\sum_i d_iV_i$. We then obtain,
\begin{eqnarray}
\dot{V}(x) &=& \sum_i d_i\dot{V}_i(x^i,u^i)\\
&\leq & U^T D \Gamma^2 U - Y^T D Y\\
&=& Y^T ((\Gamma P)^T D(\Gamma P)-D)Y ,
\end{eqnarray}
where $U=[(u^1)^T \ldots (u^N)^T]^T$, $Y=[(y^1)^T \ldots (y^N)^T]^T$, and the
second equality follows from the fact that $U=PY$. If $\dot{V}(x)$ is negative
semidefinite, we know from LaSalle's Invariance Principle \cite{khalil02}, and
for linear systems (\cite{rugh93}), that every trajectory converges to
the unobservable subspace of the system. Furthermore, under the
same assumption, and since each $H_i$ is observable, we conclude that each
trajectory must converge to the singleton $\{0\}$. Thus, the steady-state
\eqref{eq:SteadyState} is locally asymptotically stable, if there exists a
positive diagonal matrix $D$ such that $D - (\Gamma P)^T D(\Gamma P)$ is
positive definite. This is equivalent to the condition that $I-\Gamma P$ be a
M-matrix, see \cite[Theorem 2]{araki75}. To finalize the proof, note that the
spectra of $P\Gamma$ is the same as $\Gamma P$, and that the radial spectra
assumption implies that $I-\Gamma P$ is a nonsingular M-matrix, see
\cite[Definition 6.1.2]{berman94}.
\end{proof}

\vspace{5mm}
We now show that \eqref{eq:smallgaincondition} is equivalent to
$\rho(\overline{P}\overline{\Gamma})<1$, where
\begin{equation}\label{eq:gammabar}
\overline{\Gamma}=diag\{\gamma_1,\ldots,\gamma_r\}.
\end{equation}
This result simplifies the verification of this small gain stability condition
to the reduced system $S_R$ in \eqref{eq:sysdecomp} with interconnection matrix
$\overline{P}$.

\begin{lemma}\label{lemma:spectralradius}
Consider an equitable partition $\pi$ of $\mathcal{G}$. Then,
\begin{equation}
\rho(P\Gamma)<1 \Leftrightarrow \rho(\overline{P}\overline{\Gamma})<1,
\end{equation}
where $\overline{P}$ is as in \eqref{eq:overlineP}, and $\overline{\Gamma}$ is
as in \eqref{eq:gammabar}.
\end{lemma}
\begin{proof}
To prove this statement we only need to show that
\begin{equation}
\rho(\overline{P}\overline{\Gamma})=\rho(P\Gamma).
\end{equation}
First, note that $\overline{P}\overline{\Gamma}$ is a nonnegative irreducible
matrix, by the Perron-Frobenius Theorem \cite{berman94}, we know that
$r=\rho(\overline{P}\overline{\Gamma})>0$ is an eigenvalue of
$\overline{P}\overline{\Gamma}$ with corresponding eigenvector
$\overline{v}\gg 0$.\\
\textit{Claim: $r>0$ is also an eigenvalue of $P\Gamma$ with corresponding
eigenvector $v$ such that entries $v_i=\overline{v}_j$ if $i\in O_j$.}\\
According to this claim, we know that $v$ is a positive eigenvector. Therefore,
by citing again the Perron-Frobenius Theorem, and since $P\Gamma$ is also a
nonnegative irreducible matrix, we conclude that $v$ has to be the eigenvector
associated with eigenvalue $r=\rho(P\Gamma)$.\\
To prove the claim, note that matrix $\Gamma$ is positive diagonal with
repeated entries for vertices in the same class. Therefore, since the vertex
partition is equitable for the scaled adjacency graph $P$, then it is
also equitable when we consider a modified adjacency graph $P\Gamma$,
\textit{i.e.},
\begin{equation}\label{eq:overlinePGamma}
\sum_{v\in O_j}p_{uv}\gamma_{jj}=\overline{p}_{ij}\gamma_{jj}\quad
\forall u\in O_i,
\end{equation}
where $\overline{p}_{ij}$ is as defined in \eqref{eq:overlineP}. From this
observation we see that this claim is a generalization of the Lifting
Proposition in \cite{boyd05}, which holds not only for partitions obtained
through an automorphism subgroup but also for any equitable partition. The proof
of the claim follows similarly to the proof of the Lifting Proposition, with
matrices $P\Gamma$ and $\overline{P}\overline{\Gamma}$ as in
\eqref{eq:overlinePGamma}.
\end{proof}

\vspace{5mm}
Lemma \ref{lemma:spectralradius} leads to a simplification of the
stability condition \eqref{eq:smallgaincondition} in Theorem \ref{thm:stabilitytheorem}.
\begin{corollary}\label{thm:smallgaincondreduced}
Consider the network \eqref{eq:celldynamics1}-\eqref{eq:interconnection}. The
steady state pattern defined by \eqref{eq:SteadyState} is
locally asymptotically stable if
\begin{equation}\label{eq:smallgaincondreduced}
\rho(\overline{P}\overline{\Gamma})<1.
\end{equation}\hfill $\blacksquare$
\end{corollary}

\subsection{Special Case: Bipartite Graph}
Let us consider the special case of a bipartite graph, with a partition $\pi$
consisting of two classes, chosen so that no two vertices in the same set are
adjacent. As discussed before, in Example 1, the quotient matrix $\overline{P}$
is given by
\begin{equation}
\overline{P}=\left[\begin{array}{cc} 0 & 1\\ 1 & 0 \end{array} \right].
\end{equation}
Therefore, the spectra of $\overline{P}\overline{\Gamma}$ is $\pm
\sqrt{\gamma_1\gamma_2}$.
The next result follows trivially from Corollary
\ref{thm:smallgaincondreduced}.\\
\begin{corollary}
Assume that the contact graph $\mathcal{G}$ of the cell network
\eqref{eq:celldynamics1}-\eqref{eq:interconnection} is bipartite, and that there
exists a steady state $u\in R^N$ such that
$u_i=z_1$ if $i\in O_1$ and $u_i=z_2$ if $i\in O_2$, with $z_1\neq z_2\neq
z^*$, as in Corollary \ref{cor:mainBipartite}. Then, the steady state solution
is locally asymptotically stable if
\begin{equation}\label{eq:bipartitestability}
\gamma_1\gamma_2<1,
\end{equation}
where $\gamma_1$ and $\gamma_2$ are the $\mathcal{L}_2$ gains of the linearized
subsystems around $z_1$ and $z_2$, respectively. \hfill $\blacksquare$
\end{corollary}

\vspace{5mm}
In the particular case where the $\mathcal{L}_2$-gain is given by
the dc-gain,
\begin{equation}
\gamma_i=-C^i (A^i)^{-1}B^i=-T'(z_i),
\end{equation}
we see that the local asymptotic stability condition in
\eqref{eq:bipartitestability} reduces to
\begin{equation}\label{eq:monotonestabcond}
T'(z_1)T'(z_2)<1.
\end{equation}
The $\mathcal{L}_2$-gain is indeed equal to the dc-gain in the particular case
where each subsystem \eqref{eq:celldynamics1} is input-output monotone
\cite{angeli04b}, as assumed in \cite{arcak12b}. We have thus recovered Theorem
2 in \cite{arcak12b} which used \eqref{eq:monotonestabcond} to prove the
existence of stable checkerboard patterns.\\
Unlike the proof in \cite{arcak12b}, which relies heavily on monotonicity
properties, here we have only assumed that the $\mathcal{L}_2$-gain is
equal to the dc-gain.

\section{Conclusions}
In this paper we presented analytical results to predict
steady-state patterns for large-scale lateral
inhibition systems. We have shown that equitable partitions provide templates
for steady state pattern candidates, as they identify invariant subspaces where
the fate of cells in the same class is identical. We proved the existence of
steady state patterns by relying on the static input-output model of each cell and the
algebraic properties of the contact graph.
One
limitation in these results is the assumption that the reduced graph is
bipartite. Therefore, the generalization to a larger class of graph partitions,
that do not necessarily result in bipartite reduced graphs, needs to be
investigated. 
Further results will also focus on the case where the cell model is
multiple-input multiple-output, representing cell-to-cell inhibition signals
that depend on more than one species.\\

Finally, we have analyzed the stability of
steady state patterns by providing a decomposition into a representative
subsystem $S_R$ and a transverse subsystem $S_D$. We also provide a small-gain
stability type criterion, which relies only on the reduced order subsystem $S_R$
to guarantee stability of the steady state patterns.

\section*{Acknowledgment}
The authors would like to thank Katya Nepomnyaschchaya (University of
Washington) for generating Matlab functions that were used in Examples $2$, $3$,
and $4$ (for the analysis of graph symmetries, subgroups, unique orbit
partitions, and verification of conditions for the existence of pattern candidates).

\addtolength{\textheight}{-3cm}

\bibliographystyle{ieeetr}
\bibliography{researchbib}
\end{document}